\documentclass{aims}
\usepackage{amsmath}
  \usepackage{paralist}
  \usepackage{graphics} %% add this and next lines if pictures should be in esp format
  \usepackage{epsfig} %For pictures: screened artwork should be set up with an 85 or 100 line screen
\usepackage{graphicx}  \usepackage{epstopdf}%This is to transfer .eps figure to .pdf figure; please compile your paper using PDFLeTex or PDFTeXify.
 \usepackage[colorlinks=true]{hyperref}
\hypersetup{urlcolor=blue, citecolor=red}

\def\currentvolume{X}
 \def\currentissue{X}
  \def\currentyear{200X}
   \def\currentmonth{XX}
    \def\ppages{X--XX}
     \def\DOI{10.3934/xx.xx.xx.xx}

\newtheorem{theorem}{Theorem}[section]
\newtheorem{corollary}{Corollary}[section]

\newtheorem{lemma}{Lemma}[section]
\newtheorem{proposition}{Proposition}[section]

\theoremstyle{definition}
\newtheorem{definition}{Definition}[section]
\newtheorem{remark}{Remark}[section]

\DeclareMathOperator{\R}{\mathbb{R}}

\DeclareMathOperator{\N}{\mathbb{N}}

\newcommand{\I}{{\mathcal I}}

\newcommand{\cI}{{\mathcal I}}

\newcommand{\cA}{{\mathcal A}}

\renewcommand{\phi}{\varphi}

\newtheorem{teo}{\textbf{Theorem}}[section]
\newtheorem{example}{\textit{Example}}[section]

\title[Independent sub-domains Reconstruction] %Use the shortened version of the full title
{Reconstruction of independent sub-domains for a class of Hamilton Jacobi equations and its application to parallel computing}

\author[Adriano Festa]{}


 \email{festa@ensta.fr}

\begin{document}
\maketitle

% Enter the first author's name and address:
\centerline{\scshape Adriano Festa }
\medskip
{\footnotesize
% please put the address of the first author
 \centerline{ENSTA ParisTech}
   \centerline{828, Boulevard des Mar\'echaux,}
   \centerline{91120 Palaiseau, FRANCE}
} % Do not forget to end the {\footnotesize by the sign }
\bigskip

% The name of the associate editor will be entered by an editorial staff
% "Communicated by the associate editor name" is not needed for special issue.
% \centerline{(Communicated by the associate editor name)}

%The abstract of your paper
\begin{abstract}
A previous knowledge of the domains of dependence of an Hamilton Jacobi equation can be useful in its study and approximation. Information of this nature are, in general, difficult to obtain directly from the data of the problem. In this paper we introduce formally the concept of \emph{independent sub-domains}  discussing their main properties and  we provide a constructive implicit representation formula. Using such results we propose an algorithm for the approximation of these sets that will be shown to be relevant in parallel computing of the solution.
\end{abstract}

\section{Introduction}
A classic, powerful approach to optimal control problems consists in solving a verification partial differential equation in the Hamilton Jacobi form obtained using the Bellman's Dynamic Programming principle. One remarkable advantage of this approach, compared to the optimality conditions study, is the ability to provide global minima and closed-loop optimal controls; on the other hand, the study and the approximation of the value function of the problem which is an  unavoidable technical step, is often difficult. An exceptional achievement was made with the introduction of viscosity solutions, a weak notion of solution proposed by Crandall, Evans and Lions in the 80s, and the successive refinements (for a whole presentation of this subject referring to the monographs \cite{BarCap97,Bar94a}). A special benefit of this approach is to move the attention from the study of the trajectories of the problem to some geometric properties of the value function, opening in this way a wide range of techniques that, in the following, have shown a good efficiency. \par
In this paper we will consider a problem related to this. It consists in the detection of a collection of subsets, contained in the domain of the problem, where the value function restricted to each sub-domain is independent on the value of any other subset. This knowledge is useful for several reasons: it is related to stabilization problems, reachability sets reconstruction, or, as we will show later in this paper, in parallel procedures for a fast numerical resolution. This point is of special interest. In fact it is well known as the greatest limitation to the use of the Bellman approach in optimal control is due to the complexity of solving the Hamilton Jacobi equation associated; especially in high dimensional context. \par
The ``curse of dimensionality'', a meaningful definition proposed by the same R. Bellman, has been attacked in the last twenty years from several directions and various tools have been developed. Some examples of acceleration techniques are Fast Marching \cite{Set99,tsitsiklis1995} and Fast Sweeping methods \cite{Zhao05}. The strategy of these methods is to focus on the implicit order driven by the characteristics of the problem to compute each node only once, reaching convergence to the discrete solution in finite time. Despite the high efficiency of these techniques, born in the context of the resolution of the Eikonal case (where the organisation of the nodes is more evident), a generalization to more general equations is not trivial and still under investigation. Some proposals can be found  (Fast Matching  methods) in \cite{carlini2008, sethian2003ordered, cacace2011local} and (generalized Fast Sweeping methods) \cite{tsai2003fast}.\par
Another possible strategy is to decompose the domain in a collection of subsets, chosen in a number which is sufficient to lower the quantity of nodes to process in every sub-problem. Meanwhile to solve in parallel on every sub-domain. The difficulty in this idea is about  conditions to impose on the interface between two different subsets, or equivalently if some regions of overlapping are introduced, about the manner to handle the computation. Moreover, a technique of that kind requires an iterative process performed on every subdomain, with a consequent growth of the total complexity. For a whole dissertation on the subject of domain decomposition techniques (DD) we refer to the monograph \cite{valli1999domain},  for Hamilton Jacobi equations \cite{camilli1994domain, sun1993domain} and about a parallel version of the Fast Sweeping Methods \cite{Zhao07, DMG13}. \par
An alternative approach was proposed in \cite{zhou2003new} where the authors, passing to a quasi variational inequality formulation which is shown to be equivalent of the original problem, can handle a decomposition of the domain.\par
A new direction of research was opened by Ancona and Bressan in \cite{AncBre99} where they introduced the original concept of \emph{patchy feedback} with the intention of studying an asymptotic stability problem. Navasca and Krener in \cite{NavKre07} used these ideas to develop  a technique of reconstruction for the feedback solution in some special polynomial cases (\emph{patchy solutions}). Again, these elements inspired the work of Cacace et al. \cite{cacace2012patchy}. They propose, in a special class of Hamilton-Jacobi equations, a preliminary procedure called \emph{patchy decomposition}. This preliminary computation supplies a partition of the domain in sets which could be computed separately without any exchange of information between the interfaces. The result is archived using the multi-grid idea of pre-computing the problem on a coarse grid, solving the synthesis problem to have an optimal feedback control, and using it to detect a decomposition of the domain, accordingly with an approximation of the characteristics of the problem. They show in practice that the error added in this procedure is sufficiently small in some cases of interest. Our paper can be considered a development of this idea. By using some recent results in decomposition techniques we will state in a rigorous way the concept of \emph{independent sub-domain} (different concept from \emph{patchy subset}) and we will show an easy property which permit us an implicit way to reconstruct them without the delicate step about the feedback control  (we recall that, in general, there is no guarantee of convergence of an approximated optimal feedback to the continuous control). We will be able to prove the convergence of the technique and provide some error estimates. With this background we will enlarge significantly, with respect to the tests performed in \cite{cacace2012patchy}, the class of equations where the decomposition is appropriate. Indeed, through the paper there will be discussed the analogies and the differences with the \emph{patchy decomposition}.\par
The paper is organized as it follows: in Section 2 we will introduce, functionally to our purposes, the minimum property which is useful for the decomposition and the concept of \emph{independent sub-domains}. In Section 3 we propose an algorithm for their location. Main result of this part is the necessary condition contained in Proposition \ref{p:1} which characterizes the points of the grid belonging to a certain independent domain. Finally in Section 4 there will be discussed the application of the previous results to propose a parallel algorithm for the approximation  of the solution. The main benefit of our proposal will be claimed in Proposition \ref{conv} where the convergence of the technique and a bound for the error will be proved. Through some test examples we will show the good features of the proposal.

\section{Formulation and a decomposition property}
In this section there are presented some basic definitions and features which are necessary for the comprehension of the paper. In particular the decomposition property presented in details in \cite{FV14,FV14a} will be shortly recalled. \par
Let us first of all introduce the classical framework of an \emph{exit problem}. We will refer to the general structure of a \emph{differential game}. A generic \emph{optimal control problem} can be viewed as sub-case.\par

Let the dynamics be given by
\begin{equation*}
\left\{ \begin{array}{l}
\dot{y}(t)=f(y(t), a(t),b(t)), \quad a.e.\\
y(0)=x,\end{array}\right.
\end{equation*}
where $x\in\Omega$ is an open subset of $\R^n$, $a\in \mathcal{A} :=\{a:\R^+\rightarrow A, \,measureable\}$, and $b\in \mathcal{B} :=\{b:\R^+\rightarrow B, \,measureable\}$ with $A,B$ compact sets of $\R^m$. We take $f$ Lipschitz continuous with respect to the first variable and continuous with respect to $(x,a,b)$; this is enough to guarantee the existence of a solution $y_x(t,a(t),b(t))$ which will be called \emph{trajectory}. \\

%We assume moreover  [CHECK! NOT NECESSARY]
%\begin{equation}\label{fbound}
%|f(x,a)|\leq M_f \quad \forall x\in\Omega, \; \forall a\in A
%\end{equation}
%this is not always necessary but it is useful to simplify the presentation. 

%For weaker conditions to guarantee the well-posedness and Lipschitz continuity of the solution see \cite{BarCap97, CanSin04}.\\

The goal is to find the optimum (a $\sup-\inf$ optimum) over $\cA$, $\mathcal{B}$ of the functional
\begin{multline*} 
J_x(a,b):=\int_0^{\tau_x(a,b)}l(y_x(s,a(s),b(s)),a(s),b(s))e^{-\lambda s} ds\\ +e^{-\lambda \tau_x(a,b)} g(y_x(\tau_x(a,b))),\quad \lambda\geq 0,
\end{multline*}
where $\tau$ is the \emph{time of the first exit from the set $\Omega$} defined as 
$$  \tau_x(a,b):=\min\left\{t\in[0,+\infty)\;|\; y_x(t,a(t),b(t))\notin \Omega\right\}, $$
for the continuity of the trajectories $y_x(\tau_x,a(\tau_x),b(\tau_x))\in\overline{\Omega}$. \par
Typical hypothesis on the data, are stated: calling $l$ the \emph{running cost}, and  $g$ the \emph{exit cost}:
\begin{equation}\label{H0} \tag{H0}
\left.
\begin{array}{ll}
f:(\Omega,A,B)\rightarrow \R, & \text{ continuous function},\\
 & \text{ Lipschitz continuous in the first variable }\\
l:(\Omega,A,B)\rightarrow (\rho, +\infty], &\text{ is a strictly positive continuous function,}\\
&\text{ Lipschitz continuous in the first variable,}\\
g:\bar{\Omega}\rightarrow \R &\text{ is a continuous function.}
\end{array}\right\}
\end{equation}
Using the Elliot-Kalton's notion \cite{elliott1972existence} of \emph{non anticipating strategies}, we define the value function of this problem as
\begin{equation}\label{P1}  v(x):=\sup_{\phi\in\Phi}\inf_{a\in\mathcal{A}} J_x(a,\phi(a)),\end{equation}
where 
\begin{multline*}
\Phi:=\{\phi:\mathcal{A}\rightarrow \mathcal{B}: t>0, a(s)=\tilde{a}(s) \hbox{ for all }s\leq t \\ \hbox{ implies } \phi[a](s)=\phi[\tilde{a}](s) \hbox{ for all }s\leq t \}.
\end{multline*}
For a simpler presentation, we will assume the \emph{Isaacs' conditions} verified, then the value function of the problem exists, is unique and coincides with $v$. It is well known that such function  is a \emph{viscosity solution} of the problem
\begin{equation}\label{HJ}
\left\{
\begin{array}{ll}
\lambda v(x)+H\left(x,Dv(x)\right)=0& x\in\Omega\\
v(x)=g(x) & x\in\Gamma
\end{array} \right.
\end{equation}
 where the \emph{Hamiltonian} is defined as $H(x,p):=\min_{b\in\mathcal{B}}\max_{a\in\mathcal{A}}\{-f(x,a,b)\cdot p-l(x,a,b)\}$, and the $n-1$ dimensional set $\Gamma\in\partial \Omega$. 
To avoid a large number of technicalities and focus on our purposes, we will state as hypothesis:
\begin{equation}\label{lip}
\hbox{the problem \eqref{HJ} has an unique Lipschitz continuous viscosity solution } v(x). \tag{H1}
\end{equation}
This assumption will be essential in the following; conditions to ensure such regularity of the solution have been largely discussed in literature (just to cite some monographs \cite{BarCap97,Bar94a, CanSin04}).\par
%
%In particular, difficulties may appear on the boundary of the domain, where the viscosity solution could not connect smoothly with the boundary conditions, or in the case of presence of some parts of the domain which are not reachable. This kind of problems are generally overcome imposing some compatibility conditions on $g$ and controllability requests on the dynamic. (In our case we simply assume \eqref{lip}).
\par
A key property of the value function that we will use in the following is the possibility to solve a collection of Hamilton-Jacobi equations obtaining the original solution as the point-wise minimum of such family. This property was discussed in the work \cite{FV14}; here we report the result and the main points of the proof for the benefit of the reader. \par

Consider a decomposition of the set $\Gamma$ as a union of a collection of subsets, i.e. $\Gamma:=\bigcup_{i\in\I}\Gamma_i$, with $\I:=\{1,...m\}\subset\N$. We call $v_i:\bar{\Omega}\rightarrow \R$ a Lipschitz continuous viscosity solution of the problem
\begin{equation}\label{HJD}
\left\{
\begin{array}{ll}
\lambda v_i(x)+H\left(x,Dv_i(x)\right)=0& x\in\Omega\\
v_i(x)=g_i(x) & x\in \Gamma
\end{array} \right.
\end{equation}
where $g_i:\Gamma\rightarrow \R$ is a regular function such that 
\begin{equation}\label{gcond}
\begin{array}{l}
g_i(x)=g(x), \hbox{ if }x\in\Gamma_i,\\
g_i(x)> g(x), \hbox{ otherwise}.
\end{array}
\end{equation}

Also in this case, we will ask the existence of a Lipschitz continuous solution of every equation \eqref{HJD}.\par
The limiting superdifferential $\partial^{L}v(x)$ of the continuous function $v(\cdot)$ at $x$ is defined as:
$$
\partial^{L}v(x)\,:=\,\{p\,|\, \exists \mbox{ sequences } \;p_{i}\rightarrow p \mbox{ and } x_{i}\rightarrow x \mbox{ s.t. } p_{i}\in D^{+}v(x_{i}) \mbox{ for each } i  \}\, ,
$$
where $D^+v(x)$ is the usual Fr\'echet superdifferential.\\
The \emph{active indexes set} is stated as
\begin{eqnarray*}
&& I(x)\,=\, \{ j \in \{1,\ldots,m\}\,|\,  v_{j}(x)= \underset{i\in\I}{\min}\;v_{i} (x)  \}, 
\quad \mbox{for each } x \in \Omega. 
\end{eqnarray*}
We are now ready to recall the decomposition result:

\begin{teo}\label{ttt:1}
Let be verified \eqref{H0}-\eqref{lip} and the Isaacs' conditions. 
Define the set $\Upsilon\subset \Omega$ as 
$ \Upsilon:=\{x\in\Omega| \;Card(I(x))>1\}$
(where $Card(A)$ is the cardinality of the set $A$) and the function $\bar v:\Omega\rightarrow \R$ as 
$$\overline{v}(x):=\min_{i\in\I} v_i(x).$$
Under the hypothesis
\begin{equation}\label{C)}
  \lambda \bar v(x)+H(x,\sum_{i\in I(x)}\alpha_i p_i)\leq 0, \tag{H2}
\end{equation}
where $p_{i}\in \partial^{L}v_{i}(x)$ for each $i \in I(x)$, $x\in \Upsilon$, and any convex combination $\{ \alpha_{i}\,|\, i \in I(x)\}$,
we have that $\bar v$ is the unique viscosity solution of the problem \eqref{HJ}.
\end{teo}

\begin{proof}
We know that $\overline{v}$ always verifies the boundary conditions from the definition of value function and \eqref{gcond}. 
We show now that $\overline{v}$ is both subsolution and supersolution in $\Omega$ of the problem \eqref{HJ}. For uniqueness  the thesis follows. \par
The proof of $\bar v(x) $ supersolution is classic in literature. The property of subsolution is less trivial. If $x\notin \Upsilon$, i.e. $I(x)$ contains a single index value $j$, the property is directly verified; so $x\in\Upsilon$. Now, $v_{j}(\cdot)$ is Lipschitz continuous on a neighbourhood of $x$ for each $j \in I(x)$. Since $p\in D^{+}\bar v(x)$, it is certainly the case that $p \in \partial^{L} \bar{v}(x)$. Using the property that $\bar v(x')$ coincides with $\max \{ v_{j}(x')\,|\, j \in I(x')   \}$ for $x'$ in some neighbourhood of $x$, we deduce from the max rule for limiting subdifferentials of Lipschitz continuous functions  (see, e.g. \cite{Vinter00}) applied to $-\bar v(\cdot)$ the following representation for $p$:
$$
p =\sum_{j\in I(x)}\alpha_{j}p_{j}\,,
$$
for some convex combination $\{\alpha_{j}\,|\, j \in I(x)\}$ and vectors $p_{j} \in \partial^{L}v_{j}(x)$, $j\in I(x)$. But then, by  \eqref{C)},
$$
\lambda \bar v(x)+H(x,p)= \lambda \bar v(x)+H\left(x,\sum_{j \in I(x)}\alpha_{j} p_{j}\right) 
%\leq \sum_{j \in I(x)}\lambda_{j} F(x,u_{j}, p_{j})
\,\leq 0\,.
$$
  This shows that $\overline{u}(x)$ is a subsolution and concludes the proof.
  \end{proof}

\begin{remark}
 It is quite direct to show that the request \eqref{C)} is always verified with the presence of a convex Hamiltonian. As consequence any optimal control problem is  included in our framework (in a optimal control problem the Hamiltonian associated is always convex). To pass to this special case, it is sufficient to restrict the set $B$ to a singleton.
\end{remark}

Let us now define the concept of independent sub-domains.

\begin{definition}
A  subset $\Sigma\subseteq\overline{\Omega}$ is an \emph{independent sub-domain} of the problem \eqref{P1} if, given a point $x\in\Sigma$ and an optimal control $(\overline{a}(t),\bar \phi(\overline{a}(t))$ (i.e. $J_x(\overline{a},\bar \phi(\bar{a}))\leq J_x(a,\bar \phi(a))$ for every choice of $a\in\cA$, and $J_x(\overline{a},\bar \phi(\bar{a}))\geq J_x(\bar a,\phi(\bar a))$ for any $\phi\in\Phi$), the trajectory $y_x(\bar a(t),\bar \phi(\bar a(t)))\in\Sigma$ for $t\in[0,\tau_x(\bar a, \bar \phi(\bar a))]$.
\end{definition}

It is possible to establish a link between the decomposition result and the concept of independent sub-domain. In particular we show that Theorem \ref{ttt:1} provides a constructive way to build a independent sub-domains  decomposition of $\Omega$.

\begin{proposition}\label{pp:1}
Let be verified \eqref{H0}, \eqref{lip}, \eqref{C)} and the Isaacs' conditions.
Given a collection of $n-1$ dimensional subsets $\{\Gamma_i\}_{i=1,...,m}$ such that $\Gamma=\cup_{i=1}^m \Gamma_i$, the sets defined as 
\begin{equation}\label{sigmadef}
\Sigma_i:=\left\{x\in\overline{\Omega}\;|\; v_i(x)=v(x)\right\}, \quad i=1,...,m,
\end{equation}
where $v_i$ and $v$ are defined accordingly to Theorem \eqref{ttt:1}, are independent sub-domains of the problem \eqref{P1}.
\end{proposition}

\begin{proof}
The proof obtained for contradiction, is made using the Dynamical Programming Principle (cf. \cite{BarCap97}).\\
For a fixed $i$ consider a point $x\in\Sigma_i$. Let us then assume the trajectory, for an optimal control $(\bar a,\bar \phi(\bar a))$ for the original problem, $y_x(\overline{a}(\overline{t}),\bar\phi(\overline{a}))=\overline{x}\notin\Sigma_i$ for a certain $\overline{t}\in[0,\tau_x(\overline{a}(\overline{t}),\bar\phi(\overline{a}))]$. If $\overline{t}=0$ the contradiction comes directly from the definition of $\Sigma_i$. If $\overline{t}>0$  we recall (\emph{Dynamical Programming Principle}) 
\begin{multline*}
v(x) =\sup_{\phi\in\Phi}\inf_{a\in\cA} \left\{\int_0^{\overline{t}}l(y_x(a(s),\phi(a(s))),a(s),\phi(a(s)))e^{-\lambda s} ds \right.\\ \left.+e^{-\lambda \overline{t}} v\left(y_x(a(\overline{t}),\phi(a(\bar{t})))\right)\right\},
\end{multline*}
an analogue formula is obviously valid also for $v_i(x)$. Recalling $v_i(\overline{x})>v(\overline{x})$, 
\begin{multline*}
v_i(x)=\\
\sup_{\phi\in\Phi}\inf_{a\in\cA} \left\{\int_0^{\overline{t}}l(y_x(a(s),\phi(a(s))),a(s),\phi(a(s)))e^{-\lambda s} ds +e^{-\lambda \overline{t}} v_i\left(y_x(a(\overline{t}),\phi(a(\bar{t}))\right)\right\}\\
>\inf_{a\in\cA}\left\{\int_0^{\overline{t}}l(y_x(a(s),\bar\phi( a(s))),{a}(s),\bar\phi(a(s)))e^{-\lambda s} ds 
+e^{-\lambda \overline{t}} v_i\left(y_x({a}(s),\bar\phi( a(s))\right)\right\} \\
= \int_0^{\overline{t}}l(y_x(\bar a(s),\bar\phi( \bar a(s))),{\bar a}(s),\bar\phi(\bar a(s)))e^{-\lambda s} ds +e^{-\lambda \overline{t}} v_i\left(\bar x\right)\\
\geq \left\{\int_0^{\overline{t}}l(y_x(\bar{a}(s),\bar\phi(\bar a(s))),\bar{a}(s))e^{-\lambda s} ds +e^{-\lambda \overline{t}} v(\bar x)\right\}=v(x),\end{multline*}
then $v_i(x)>v(x)$, which contradicts again the definition \eqref{sigmadef}.
\end{proof}
\par
That property of the trajectories will play an important role in the following; it will guarantee the absence of crossing information through the boundary of every independent sub-domain, or using different words, the solution of the problem \eqref{HJ} in each sub-domain will not depend on the solution in other sub-domains.\par
A feature easy to derive from Proposition \ref{pp:1} is the connexion of the sets:
\begin{corollary}
Let be verified \eqref{H0}, \eqref{lip} and \eqref{C)} and the Isaacs' conditions.
If $\Gamma_i$ is connected, the respective set $\Sigma_i$ defined in \eqref{sigmadef} is also  connected.
\end{corollary}
\begin{proof}
  If $\Gamma_i$ is connected we can always join two points $x$, $y$ of the set using the respective optimal trajectories, which, for Proposition \ref{pp:1} are contained in the set.
\end{proof}

\begin{figure}[t]\label{f:ex}
\begin{center}
\includegraphics[height=6cm]{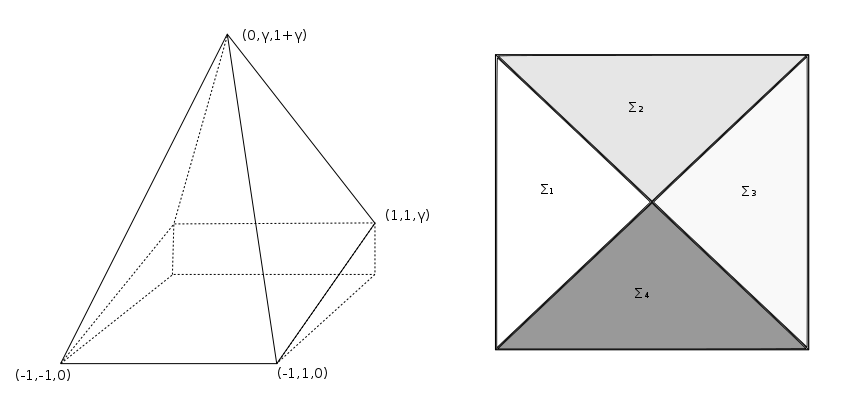}
\caption{Example \ref{ex-1}, the auxiliary solution $u_1$ and the independent sub-domains decomposition.}
\end{center}
\end{figure}

Let us to give a simple example of an independent sub-domains decomposition:
\begin{example}\label{ex-1}

The equation considered is, with $\Omega:=(-1,1)\times(-1,1)$,
\begin{equation*}
\left\{
\begin{array}{ll}
\max_{a\in B(0,1)}\{a\cdot Dv(x)\}=1 & x\in \Omega\\
v(x)=0 & x_1\in\partial \Omega.
\end{array} \right.
\end{equation*}
 Stated $\partial\Omega=\Gamma:=\cup_i \Gamma_i=\cup [(\pm 1 ,\pm 1),(\pm 1,\pm 1)]$, we associate to the $\Gamma_1:=[(-1,-1), (-1,1)]$ the function $g_1:\partial\Omega\rightarrow \R$ defined as
\begin{equation*}
\left\{
\begin{array}{ll}
g_1(x):=0 & x\in \Gamma_1\\
g_1(x):=\gamma (1+x_2) & x\in\Gamma\setminus \Gamma_1,
\end{array} \right.
\end{equation*}
for a chosen $\gamma\in \R^+$. The other $g_i$s will be defined in a symmetric manner. It is possible to verify that the unique viscosity solution of such a problem is 
$$ v_1(x)=(1+\gamma)-\max(|x_1-\gamma|,|x_2|).$$ 
Finally the original value function $v(x)=1-\max(|x_1|,|x_2|)$ is recovered as $v(x)=\min\limits_{i=1,...,4}v_i(x)$. The decomposition in independent sub-domains obtained are shown in Figure \ref{f:ex}.
\end{example}

\section{Independent sub-domains reconstruction}

In this section we will introduce a numerical technique for the approximation of the independent sub-domains based on the results of the previous section. The technique is not related to a special numerical scheme, but it needs an \emph{a priori} bound for the approximation; this is necessary to have an important property of inclusion of the sets which will play a role in the successive error analysis. As example of numerical scheme we will make reference to a semiLagrangian solver, but the procedure can be easily extended to finite difference schemes, finite volumes etc. For further details about semiLagrangian techniques we refer to the monograph by Falcone and Ferretti \cite{FalFer14}.\par
Let us consider a structured grid of $\Omega$ made by a family of simplices $S_j$, such that $\overline{\Omega}\in\cup_j S_j$. Name $x_i$, $i=1, ... , N$ the nodes of the triangulation,
\begin{equation} \Delta x :=\max_j {\it diam}(S_j)
\end{equation}
the size of the mesh ($diam(S)$ is the diameter of the set $S$). Let be $G$ the set of the internal nodes of the grid and, consequently, $\partial G$ is the set of its boundary points; in the case of a bounded $\Omega$ we call also $\Phi$ the nodes corresponding to the set $\R^n\setminus\overline{\Omega}$, those nodes typically act as \emph{ghost nodes}. We remark that this discretization space includes the classical case of  regular meshes. \par
We map all the values at the nodes in  $V=(V(1),...,V(N))$. By a standard semiLagrangian discretization \cite{BarFal90, FalFer14} of \eqref{HJ}, it is possible to obtain the following scheme in fixed point form
\begin{equation}\label{SL}
V=T(V),
\end{equation}
where $F:\R^N\rightarrow \R^N$ is defined component-wise by 
\begin{equation*}\label{Tdef}
[T(V)]_i=\left\{
\begin{array}{ll}
\max\limits_{b\in B}\min\limits_{a\in A}\left\{\frac{1}{1+\lambda h}\mathbb{I}[V](x_i-hf(x_i,a,b))-hl(x_i,a,b) \right\}& x_i\in G,\\
g(x_i) & x_i\in \partial G,\\
+\infty & x_i\in\Phi.
\end{array} \right.
\end{equation*}

The discrete value function $V$ is extended on the whole space $\Omega$ by a linear $n-$dimensional interpolation, represented by the operator $\mathbb{I}$, as described in \cite{FalFer98,Bardi1999numerical}.\\
The variable $h$ corresponds to a \emph{fictitious} time discretization with the purpose to imitate the behaviour of the characteristics of the problem. The minimum over $A$ and the maximum over $B$ is evaluated by direct comparison using a discrete version of the control space $A,B$. Generally the fixed point of the equation \eqref{SL} is found through the iterative map $V^{n+1}:=T(V^n)$ which is shown to be a contraction.\par
It is important to recall the following  result of convergence for the semiLagrangian scheme. The proof can be found in \cite{Sor98,Bardi1999numerical} for the case of differential games and in \cite{CapIsh84,Fal87} for optimal control problems. 

\begin{theorem}\label{t:1}
Let $v$ and $V$ be the solutions of, respectively, equation \eqref{HJ} and \eqref{SL}. Assume verified \eqref{H0} and \eqref{lip} then
$$ ||v-V||_{\infty}\leq C (\Delta x)^q, $$
where $C$ is a positive constant independent from $\Delta x$,  $q\in \R^+$ depending on the regularity of the problem. 
\end{theorem}

For differential games with a Lipschitz continuous solution,  a possible estimate is 
$$ \| v-V\|_{\infty}\leq C h^{\frac{1}{2}}\left(1+\left(\frac{\Delta x}{h}\right)^2\right).$$
If the quantity $\frac{\Delta x}{h}=1$, we have the relation described in Theorem \ref{t:1} with $q=1/2$.
The constant $C$ depends on the data of the problem and can be estimated. \par
In the case of an optimal control problem with $\lambda>0$, a possible convergence bound  is the following 
$$  \| v-V\|_{\infty}\leq 2(M_v+M_{v_h}) h^{\frac{1}{2}}+\left(\frac{L_l}{\lambda(\lambda-L_f)}\frac{\Delta x}{h}\right)$$
with $M_v,M_{v_h}$ maxima of the absolute value of the continuous  and semidiscrete solution and $L_l,L_f$ Lipschitz constants of dynamic and running cost. Then, in this case, for $h^2=\Delta x^3$, $C=2(M_v+M_{v_h})+\frac{L_l}{\lambda (\lambda-L_f)}$ and $q=\frac{1}{3}$.\par
Other examples of error estimates can be found in literature, even of high order (i.e. $q>1$) in some smooth cases \cite{FalFer14}. 
\par

Using the numerical scheme described above we can obtain an approximation of the solution 
of every decomposed problem \eqref{HJD}; these discrete solutions are called, in analogy with the continuous case, $V_i$ for $i\in\cI$. \par

A simple observation brings us to the following Lemma:
\begin{lemma}\label{lemma:1}
Let be verified \eqref{H0}, \eqref{lip}, \eqref{C)} and the Isaacs' conditions.
If a node $x_j\in\Sigma_i$ then there exists a $C>0$ independent from $\Delta x$ and a $q\in\R^+$ s.t. $|V_i(j)-v(x_j)|\leq C(\Delta x)^q$. The parameters $C$ and $q$ are the same than in Theorem \ref{t:1}.
\end{lemma}
\begin{proof}
It is sufficient observe that  $|V_i(j)-v(x_j)|\leq |V_i(j)-v_i(x_j)|+|v_i(x_j)-v(x_j)|$. Proposition \ref{pp:1} and Theorem \ref{t:1} give the estimate.
\end{proof}

We can establish a necessary condition for the nodes of $G$ to belong to a fixed independent sub-domain $\Sigma_i$. Let be $B(x,\rho)$ the $n-$dimensional ball centred in $x$ and of radius $\rho$.

\begin{proposition}\label{p:1}
Assume \eqref{H0}, \eqref{lip}, \eqref{C)} and the Isaacs' conditions.
Let be $x_j\in G$ such that, taken an  $\epsilon\in [0, \Delta x)$ and a direction $d\in B(0,1)$, the point $x=x_j-\epsilon d\in\Omega$ verifies $v_i(x)=v(x)$ for a certain $i\in\I$. Then the following estimate holds
\begin{equation}\label{rep}
|V_i(x_j)-V(x_j)|\leq 2(C(\Delta x)^q+M\Delta x)
\end{equation}
$C$ as in the previous statement and $M:=\max\{L_{v_i}, i\in\cI\}$ where $L_{v_i}$ is the Lipschitz constant of the function $v_i$.
\end{proposition}

\begin{proof}
It is sufficient to observe that 
\begin{multline*} |V_i(x_j)-V(x_j)|\leq |V_i(x_j)-v_i(x_j)|+|v_i(x_j)-v_i(x)|+|v_i(x)-v(x)|\\+|v(x)-v(x_j)|+|v(x_j)-V(x_j)|
\leq \|V_i(x_j)-v_i(x_j)\|_\infty+\|v(x_j)-V(x_j)\|_\infty\\+|v_i(x_j)-v_i(x)|+|v_i(x)-v(x)|
+|v(x)-v(x_j)|\leq 2C(\Delta x)^q +2 M\Delta x+|v_i(x)-v(x)|
\end{multline*}
since $v$ and $v_i$ are Lipschitz continuous with Lipschitz constant bounded by $M$. From $v_i(x)=v(x)$, \eqref{rep} follows.
\end{proof}

The previous condition gives us a necessary condition for the nodes to belong to a independent subset $\Omega_i$. Note that such condition is verified by the nodes lying in the interior of such set but also by a neighbourhood of the boundary, of thickness depending by the parameters $C$ and $M$. This criteria will be used in the invariant sub-domains reconstruction algorithm; the list of the nodes of $G$ belonging to the approximation of the independent sub-domain $\Sigma_i$ will be denoted  $\overline{\Sigma}_i$. Consequently, the relative approximated set will be the region delimited by $\overline{\Sigma}_i$ (concave hull).\par
Let us define $union(X_1,X_2)$ the vector composed by all the elements present in $X_1$ and $X_2$ and let us call $X_\Gamma$ the vector containing all the indices of the nodes in $G$ corresponding to $\Gamma$. We call also $V(j)$ the $j-$component of the vector $V$.

\bigskip
\begin{center}
%\bigskip
{\sc Indipendent Sub-Domains Reconstruction Algorithm (RA).}
\vspace{-0.2cm}
\begin{center}
\line(1,0){290}
\end{center}
\begin{itemize}
\item[- ]Given a grid $G$ of discretization step $\Delta x$ and a collection of vectors such that $union(X_i, i=1,...,m)=X_{\Gamma}$.
\item[1)] (Resolution of auxiliary problems) \\
for $i=1...m$ solve iteratively the problem \\
\hspace{1cm}$V_i=T_i(V_i)$ with $T_i$ defined as \eqref{Tdef} with $\partial G:=X_i$.\\
end
\item[2)] (Check and reconstruction of the value function)\\
If necessary, check numerically \eqref{C)}, \\ then
obtain $V$ as $V:=\min_{i=1...m}\{V_i\}$.
\item[3)] (Reconstruction of the sub-domains)\\
for i=1...m\\
\hspace{1cm} initialize $\overline{\Sigma}_i=X_i$\\
\hspace{1cm}for j=1...N\\
\hspace{2cm} if $|V_i(j)-V(j)|\leq 2(C(\Delta x)^q+M\Delta x)$ \\
\hspace{2cm}then add $x_j$ to vector $\overline{\Sigma}_i$\\
\hspace{1cm}end\\
\hspace{1cm}the $i-$subset is $\bar\Sigma_i$.\\
end
\end{itemize}
\end{center}
\bigskip
Let us underline that, from the computational point of view, the difficult step is only the first one; successive points are faster and with a negligible complexity.  In addition, point $1)$ is easily performed in parallel,  since it consists in a collection of independent problems, reducing the difficulty of resolution.

\begin{remark}
A delicate phase of the algorithm is the choice of the parameters $C$ and $M$. An error in this point would produce two opposite behaviours. Bigger parameters will bring to round up the  independent sub-domains that could even include the whole domain, making worthless the technique. A wrong choice of $C$ or $M$, smaller than the correct one, would nullify our following error analysis. In the test section (Section \ref{test}) we will show as even a not so tight choice of the parameters will produce acceptable approximations of the desired sets in most of situations of interest.
\end{remark}

\begin{remark}
  It is worth to stress an issue about the stopping criterion used in the iterative resolution \eqref{SL} contained in step $1)$. It is clear that, in general, the exact discrete solution will not be reached,  then the stopping criterion used should be compatible with our requests of accuracy. For the case of the semiLagrangian approximation, for a $\lambda>0$, the classical estimate $\|V^n-V^{n+1}\|_{\infty}\leq \frac{1}{1+\lambda h} \|V^{n-1}-V^{n-2}\|_\infty$ brings us a link between the two successive iterations and the distance (in the $L^\infty$ norm) from the discrete solution as 
$$ \|V^n-V\|_{\infty}\leq \sum_{t=n}^\infty\left(\frac{1}{1+\lambda \Delta x}\right)^t\|V^n-V^{n+1}\|_{\infty}$$
then a possible stopping criterion compatible  is 
$$\|V^{n+1}-V^n\|_{\infty}\leq  \epsilon, \quad \epsilon=2\lambda \Delta x (1+\lambda \Delta x)^{n-1}(C(\Delta x)^q+M\Delta x).$$
\end{remark}

\bigskip

The RA builds an approximation of the independent sub-domains. It is guaranteed that such approximation is performed exceeding the desired set, in the sense that $\Sigma_i\subseteq \overline{\Sigma}_i$. Another important property coming from the Proposition \ref{p:1} is that, for two discretization steps $\Delta x_1$ and $\Delta x_2$ such that  $\Delta x_1 \geq \Delta x_2$, the approximate independent sub-domains of a same decomposition of $\Gamma$ have the feature
\begin{equation}\label{incap}
\overline{\Sigma}_i^{\Delta x_1}\supseteq \overline{\Sigma}_i^{\Delta x_2}
\end{equation}
where with $\overline{\Sigma}_i^{\Delta x}$ we intend the discrete independent set obtained performing the RA with discretization space step $\Delta x$. \par

A point to discuss is the relation with the decomposition technique proposed in \cite{cacace2012patchy}. Despite the analogies, in particular the idea of finding a decomposition in subdomains which preserves a certain mutual independence, in general the decomposition obtained can be slightly different. Let us show it with an example.

\begin{example}\label{ex1}
Let us consider the domain $\Omega:=(-1,1)\times \R$ the dynamics 
$$ f(x,a):= a_1, \quad \lambda:=\delta, \quad a=B_2(0,1),$$
running cost $l(x,a)\equiv 1$,  and the set $\Gamma:=\cup_{i=1}^2 \Gamma_i$, with $\Gamma_1:=\{x_1=1\}$, $\Gamma_2:=\{x_1=-1\}$. Let us impose $g:\Gamma\rightarrow \R$ null in $\Gamma_1$ and $\Gamma_2$. It is possible to check that the solution is 
\begin{equation*}
v(x)=\left\{
\begin{array}{ll}
1-\frac{e^{\delta x_1}}{e^\delta}& \hbox{ for }x_1\leq 0\\
-1+\frac{e^{\delta x_1}}{e^\delta}& \hbox{ otherwise. }
\end{array}\right.
\end{equation*}
We can notice how for every choice of $\delta$ the invariant domains relative to the two subtargets are respectively $\{x\in\Omega\;|\; x_1\geq 0\}$ and $\{x\in\Omega\;| \;x_1\leq 0\}$.  \\
\begin{figure}[t]
\begin{center}
\includegraphics[height=4.5cm]{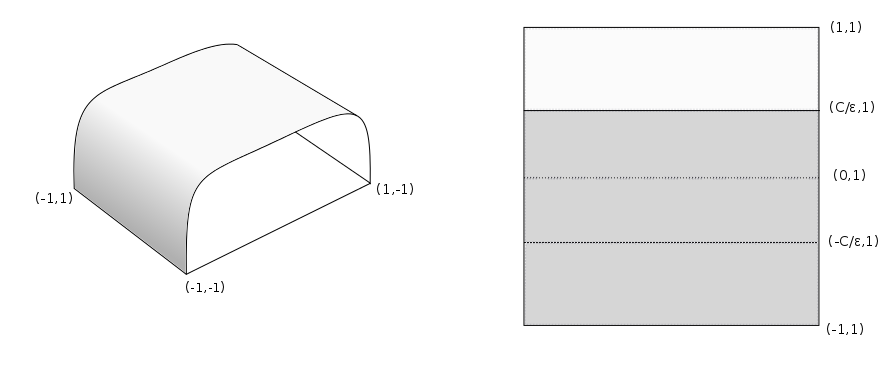}
\caption{Example \ref{ex1}, the flatness of the central region depends on $\delta$, in tones of grey the incorrect division coming from numerical incertitude.}
\end{center}
\end{figure}
For $\delta$ sufficiently small the numerics, although reconstructing correctly according to a certain error bound the approximate value function, will not be able to solve appropriately the synthesis problem. The assignment of the approximated optimal control of the region ``too flat'' will depend on the priority chosen in the computing. So using the \emph{patchy decomposition}, the sub-domains reconstructed will be, for example $\{x\in\Omega\;|\;x_1+\frac{r}{\delta}\geq 0\}$ and $\{x\in\Omega\;|\;x_1+\frac{r}{\delta}\leq 0\}$, $r\in (0,c)$, for a fixed $c$ depending on the computing parameters, which are arbitrarily different (for a generic $\delta$) from the correct division. Differently, our implicit reconstruction will produce $\{x\in\Omega\;|\;x_1+\frac{c}{\delta}\geq 0\}$ and $\{x\in\Omega\;|\;x_1-\frac{c}{\delta}\leq 0\}$ which are larger sets containing the correct decomposition. 
\end{example}

\section{Application to parallel computing} 

In this section we show as the results about the reconstruction of the independent sub-domains can be used to compute in parallel the correct solution of the discrete problem \eqref{SL}. We prove the convergence of the technique, and provide a bound for the numerical error. Roughly speaking the proposal is based on the reconstruction of a collection of independent subsets, computed in parallel on a coarse grid, and successively the computation of the solution on every sub-domain on a fine grid, recovering at the end the result, using the minimum property on the regions of overlapping.\par
Let us state more rigorously the technique.\par
Consider two families of simplices: a \emph{coarse grid} $K$ of discretization step $\Delta x_K$ and a \emph{fine grid} $G$ of step $\Delta x_G$ which both cover the domain $\Omega$, (i.e.  $\overline{\Omega}\subseteq\cup_j S_j\subseteq\cup_j K_j$), call $z_k$, $k=1, ... , N_1$, the nodes of the first triangulation and $x_k$, $k=1,...,N_2$  the nodes of the second triangulation. Triangulations are chosen such that $N_1\ll N_2$. The set of the nodes of $K$ corresponding to $\Gamma $ will be called $Z_\Gamma$.\\
The parallel invariant sub-domains based algorithm is the following:

\par
\begin{center}
\bigskip
{\sc Independent-Sets Algorithm (ISA).}
\vspace{-0.3cm}
\begin{center}
\line(1,0){300}
\end{center}
\begin{enumerate}
\item[- ]Given a grid $K$ and a collection of vectors $Z_i$ such that $$union(Z_i, i=1,...,M)=Z_\Gamma$$
\item[1)](Reconstruction of the approximated independent sub-domains)\\
     Using {\sc RA} get a collection $\overline{\Sigma}_i^{\Delta x_k}$, $i=1,...,M$ subsets of the grid $K$.
\item[2)](Projection on the fine grid)\\
Project $\overline{\Sigma}_i^{\Delta x_K}$ on the grid $G$ getting $\overline{\Sigma}^G_i$ for $i=1,...,M$,
\item[3)](Resolution on the fine grid)\\
for $i=1,...,M$ solve iteratively the problem on $\overline{\Sigma}^G_i$\\
\hspace{1cm}$V_i=T(V_i)$ with $T$ defined as \eqref{Tdef}\\
end
\item[4)](Assembly of the final Solution)\\
for $j=1,...,N_2$ \\
\hspace{1cm}$\overline{V}(j)=min\{V_i(j)|x_j\in \overline{\Sigma}^G_i \}$\\
end
\end{enumerate}

\end{center}
\bigskip

Some observations about the algorithm described above:
\begin{itemize} 
\item computing of the $RA$ at point $1)$ it is not more expansive than a single computation on the coarse grid. $RA$ is an algorithm which can work in parallel, and the number of threads that it needs, are the same requested at point $4)$.\\
\item  The projection is very easy if the grid are chosen to be partially superimposed i.e. every point $z_i\in K$ is also a point of the fine grid $G$; in every case the condition to impose is 
$$ x_j\in \overline{\Sigma}^G_i \Longleftrightarrow x_j\in Con(\{z_j\;|\;j\in \overline{\Sigma}_i^{\Delta x_K}\});$$
where $Con(\cdot)$ is the concave hull of the set, i.e. the union of the simplexes with vertexes in the set. The computational cost of this passage is negligible.
\item It is evident from definitions and from \eqref{incap} that 
\begin{equation}\label{incapsulamento} \Sigma_i\subseteq \overline{\Sigma}^{\Delta x_G}_i\subseteq \overline{\Sigma}^G_i\equiv \overline{\Sigma}^{\Delta x_K}_i. \end{equation}
\end{itemize}
This last observation will guarantee a delicate point about the convergence of the method, as we show in the following proposition:
\begin{proposition}\label{conv}
Assume \eqref{H0}, \eqref{lip}, \eqref{C)} and the Isaacs' conditions.
Called $\overline{V}$ the exact discrete solution of the $ISA$ algorithm (i.e. all $V_i=T_i(V_i)$ are verified exactly) and $V$ the exact solution of \eqref{SL} (i.e. $V=T(V)$ is exactly verified), then there exists a $C>0$ and a $q\in\R^+$ independent from $\Delta_G$ such that
$$ \|\overline{V} -v\|_{\infty}\leq C (\Delta x_G)^q$$
holds. The parameters $C$ and $q$ are the same than in Theorem \ref{t:1}.
\end{proposition}
\begin{proof}
For the observation \eqref{incapsulamento} we know that the independent sub-domains eventually obtained on the fine grid, should be subsets of the sub-domains we used in the algorithm.\\
Let us take a $x_j\in G$, through Proposition \ref{p:1} and \eqref{incapsulamento} it is assured that there exists at least one index $i\in \{1,...,M\}$  such that $v(x_j)=v_i(x_j)$ (solution of \eqref{HJD}), and $|\overline{V}(x_j)-V_i(x_j)|=0$. Then, using Lemma \ref{lemma:1}
\begin{multline*}
   |\overline{V}(x_j)-v(x_j)|\leq |\overline{V}(x_j)-V_i(x_j)|+|V_i(x_j)-v(x_j)|
\leq\|V_i-v_i\|_\infty\leq C (\Delta x_G)^q
\end{multline*}
for the arbitrariness of the choice of $x_j$ we have the thesis. 
\end{proof}

\subsection{Some examples} \label{test}
In this section we will give some examples of problems solved with ISA. 
The aim of this section is purely descriptive, and our intention is to give an experimental confirmation to the results stated previously. We will show practically that the procedure of the independent domain approximation RA is computational cheap and does not add an excessive number of nodes, even when the coarse grid $K$ consists of a  small number of control points. Together  we will verify that our technique does not add a numerical error with respect to the solution found on the whole domain and we will briefly compare the performances of our proposal to the literature. 
For a rigorous and more detailed comparison  we postpone to future works more computationally oriented.\par
Let us first of all recall the discrete analogue of the $L^\infty$, $L^1$ norms for a vector $X$ of  $N$ elements:
$$\|X\|_{\Delta_\infty} :=\max_{j=1,...,N}|X(j)|, \qquad \|X\|_{\Delta_1} :=\frac{1}{N}\sum_{j=1}^N|X(j)|.$$

\begin{example}[Distance function]

Let us to start with the very easy case shown in Example \ref{ex-1}, which will be useful to observe some general features. Therefore we  consider the Eikonal equation on the set $\Omega:=(-1,1)^2$ with the boundary value fixed to zero on $\Gamma:=\partial \Omega$. This equation models the distance from the boundary of such set. The case considered here will be with $\lambda=1$, it will correspond to a non linear monotone scaling of the solution (this relation is classically shown through the Kruzkov transform see \cite{BarCap97}), so the correct viscosity solution is the function
$$ v(x)=1-\frac{\min\{e^{|x_1|},e^{|x_2|}\}}{e},$$
solving the equation
$$ v(x)+\max_{a\in B(0,1)}\{a\cdot Dv(x)\}=1.$$

\begin{figure}[t]
\begin{center}
\includegraphics[height=5.8cm]{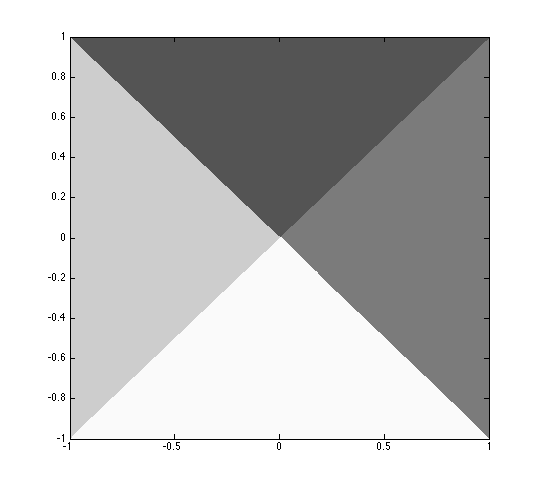}
\includegraphics[height=5.8cm]{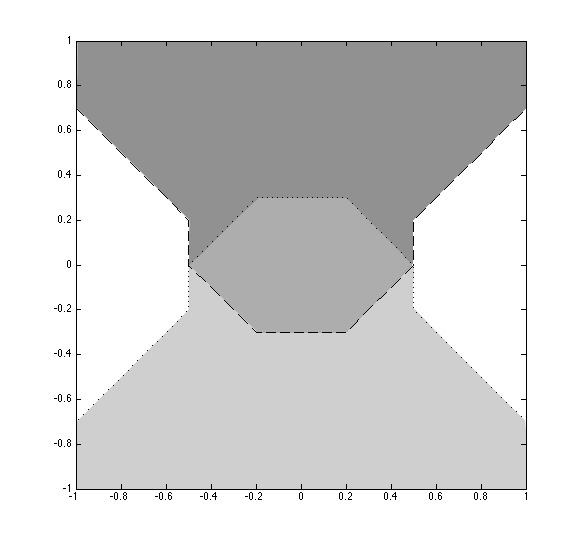}
\vspace{-0.5cm}
\caption{Distance function: exact decomposition and two (of the four) approximated independent subsets found with a coarse grid of $15^2$ points (the third tone of grey in the centre is the superposition area between the two sets).} \label{fig1}
\end{center}
\end{figure}

\begin{table}
\begin{center}
\caption{Distance function: comparison of the accuracy of the decomposition with various discretization steps}\label{tt:1}
\begin{tabular}{{c}|*{2}{c}|{c}|{c}}
         N. of variables  & $\Delta x_K$  &  Time elapsed   &   $\max_i |\overline{\Sigma}_i|/|\Omega|$ & $\max_i |\Sigma_i|/|\Omega|$\\
\hline
\bfseries $5^2$ & 0.4 &  $1\cdot 10^{-3}$ & $50\%$   &   \\
\bfseries $7^2$ & 0.28 & $2\cdot 10^{-3}$ & $43\%$ &  \\
\bfseries $10^2$ & 0.2 &  $4\cdot 10^{-3}$ &  $38\%$ &\\
\bfseries $15^2$  & 0.133&   $2\cdot 10^{-2}$  &  $35\%$  &  $25\%$\\
\bfseries $20^2$ & 0.1 & $5\cdot 10^{-2}$ & $33\%$ &  \\
\bfseries $30^2$  & 0.06 &  $1.01$ &  $30\%$  &\\
\bfseries $40^2$  & 0.05 &  $3 $  &  $29\%$  &   \\
\bfseries $50^2$ & 0.04 & $11$ & $28.3\%$ &  \\
\end{tabular}
\end{center}
\end{table}

\begin{table}
\begin{center}
\caption{Distance function: comparison between the efficiency of the various methods (ND no decomposition, DD domain decomposition, ISA Independent set decomposition)}\label{tt:2}
\begin{tabular}{*2{c}|{c}|*1{c}|{c}}
         N. of variables    &  $\Delta x_K $ & Time  ND   & Time (it) DD&  Time ISA\\
\hline
\bfseries $25^2$  & $0.08$ & $0.13$ &$0.06(2)$& $0.035$      \\
\bfseries $50^2$ & $0.04$ & $7.02$ & $1.2(2) $&$0.68$   \\
\bfseries $75^2$  & $0.026$ & $57.5$ &  $12.2(3)$&$ 4.8$ \\
\bfseries $100^2$  &  $0.02$ &$1.5\cdot 10^3$  & $65.3(3)$& $16.6$  \\
\bfseries $200^2$ & $0.01$& $1.9\cdot 10^5$ & $1.2\cdot 10^4(5)$ &$3\cdot 10^3$   \\
\bfseries $300^2$ & $0.006$& $>10^6$ & $1.8\cdot 10^5(11)$ & $4.6\cdot 10^4$   \\
\end{tabular}
\end{center}
\end{table}

\begin{table}
\begin{center}
\caption{Distance function: approximation error Error in norm $\Delta \infty$ (and $\Delta_1$) }\label{tt:3}
\begin{tabular}{{c}|*{3}{c}}
             &  $50^2 $ &  $100^2$  &    $200^2$\\
\hline
\bfseries original  &  $1.2\cdot 10^{-2} (1.1\cdot 10^{-2})$ &$6.5\cdot 10^{-3} (3.6\cdot 10^{-3})$  &  $2.5\cdot 10^{-3} (1.6\cdot 10^{-3})$  \\
\bfseries 2-subsets  &  $1.2\cdot 10^{-2} (7.2\cdot 10^{-3})$ &$6.5\cdot 10^{-3} (3.7\cdot 10^{-3})$  &  $2.5\cdot 10^{-3} (1.4\cdot 10^{-3})$     \\
\bfseries  4-subsets &  $9\cdot 10^{-3} (7.2\cdot 10^{-3})$ &$4.6\cdot 10^{-3} (3.6\cdot 10^{-3})$  &  $1.4\cdot 10^{-3} (1.3\cdot 10^{-3})$   \\
\bfseries 8-subsets  &  $9\cdot 10^{-3} (7.2\cdot 10^{-3})$ &$4.6\cdot 10^{-3} (3.6\cdot 10^{-3})$  &  $1.4\cdot 10^{-3} (1.3\cdot 10^{-3})$   \\
\end{tabular}
\end{center}
\end{table}

In the following we will refer to a uniform decomposition of the set $\Gamma$, for example in a 2-treads decomposition $\Gamma_1:=[-1,1]\times\{-1\}\cup {-1}\times [-1,1]$,  $\Gamma_2:=[-1,1]\times\{1\}\cup \{1\}\times [-1,1]$, in a 4-treads $\Gamma_1:=[-1,1]\times\{-1\}$, $\Gamma_2:=\{1\}\times[-1,1]$, $\Gamma_3:=[-1,1]\times\{1\}$, $\Gamma_2:=\{-1\}\times[-1,1]$; etc.\par
In this case it is easy also to give an estimation of the constants introduced above, $M=1$, $C=1$ and $\epsilon=10^{-3}$, $q=1/2$. It is evident the fact that the precision of the independent subdomains reconstruction will be affected by the discretization step used in the procedure. In Table \ref{tt:1} there is a comparison, in the case of a 4-subsets decomposition, of such accuracy. The percentage reported is the maximal extension of an approximated subset $\overline{\Sigma}_i$ on the total area of $\Omega$. Evidently in every case, the exact decomposition will be contained in the approximated one. It is worth to underline how, even a very coarse grid (with $10^2$ or $15^2$ elements) the technique is able to provide a sufficiently accurate estimate, giving a good reduction of the dimension of the sub problems with a cost of the precomputing section absolutely negligible. In Figure \ref{fig1} is reported the exact decomposition and two approximation  sets $\Sigma_1$, $\Sigma_3$ with $\Delta x=0.2$.\par

After the decomposition, the problem can be solved separately, and possibly at the same time on each $\overline{\Sigma}_i$.
It is consequential a large gain in term of computational cost, compared to the resolution on the whole domain. We show,  in Table \ref{tt:2}, the time of computation of the resolution in the whole domain ND, compared with a standard domain decomposition methods DD (4 equal sub-domains, no superposition, as described in \cite{camilli1994domain}) where we show also the number of iterations between the sub-domains necessary to reach the solution and our algorithm ISA.\par
As expected our performances are comparable with a single iteration of the DD, which consists in a resolution on a part of the grid containing the $25\%$ of the nodes of $\Omega$. The ISA is performed on a  approximated independent sub-domains decomposition obtained using RA on a $20^2$ grid; accordingly as shown in Table  \ref{tt:1}  the dimension of such decomposed domain will be approximately $33\%$ of the original one. \par

The most important peculiarity of this proposal consists in having a bound for the convergence of the method. Table \ref{tt:3} shows the experimental error in various \emph{fine grids}, in the case of the original problem (solved on the whole domain $\Omega$) or in presence of various decompositions in independent domains. It is important to remark again the fact that the error introduced in not affected by the discretization step  of   the \emph{coarse grid}  for the reconstruction of the subdomains. This is not shown in a table because it would be simply constant through the various choices of $\Delta x_K$. The experimental evidence shows us an error  even smaller than the normal resolution, or at least equal. This is due to the fact that in presence of a ``favourable'' decomposition of the problem, where non differentiability points of the solution lie on a regular point of every decomposed solution, the value in that point is reconstructed as a minimum of a collection of smooth functions. Numerically, this decreases the local numerical error due to the smoothing effect of the interpolation in \eqref{SL}.

\end{example}

\begin{example}[Van Der Pol oscillator]\label{ex3}
Let us pass to a more difficult case. A well known example in the field is the \emph{Van Der Pol oscillator}, here formulated as target problem. We consider $\Gamma:=\partial B(0,\rho)$ (in this case $\rho=0.2$) and $\Omega:=(-1,1)^2\setminus \bar B(0,\rho)$. Also in this case the example is an optimal control problem, then the second control will not be present in the dynamics. The  nonlinear system will be:
\begin{equation*}
f(x,a)=\left(\begin{array}{c} x_2 \\ (1-x_1^2)x_2-x_1+a \end{array} \right).
\end{equation*}
The others parameters of the system are: 
$$\Omega=(-1,1)^2, \quad A=[-1,1], \quad \lambda =1, \quad l(x,y,a)=(x_1^2+x_2^2)^\frac{1}{2}, \quad g(x)\equiv 0.$$
For this problem we do not have  an analytic formula for the solution, then we will consider, (in the error estimation), a numerical solution computed on a very fine grid of $400^2$ elements.

\begin{figure}[t]
\begin{center}
\includegraphics[height=5.8cm]{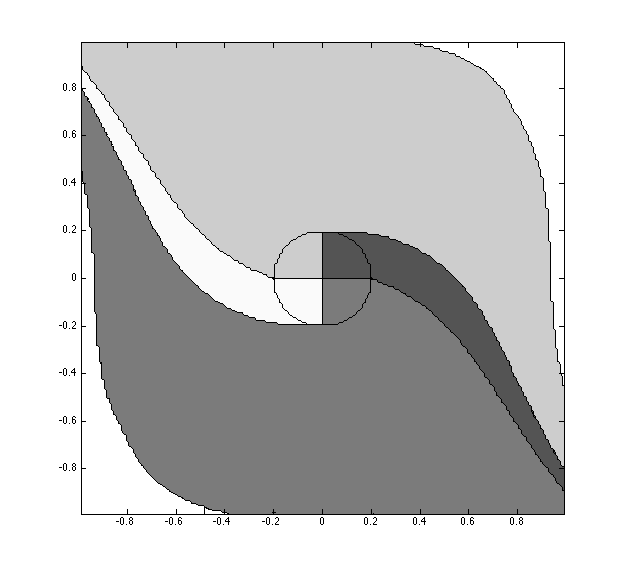}
\includegraphics[height=5.8cm]{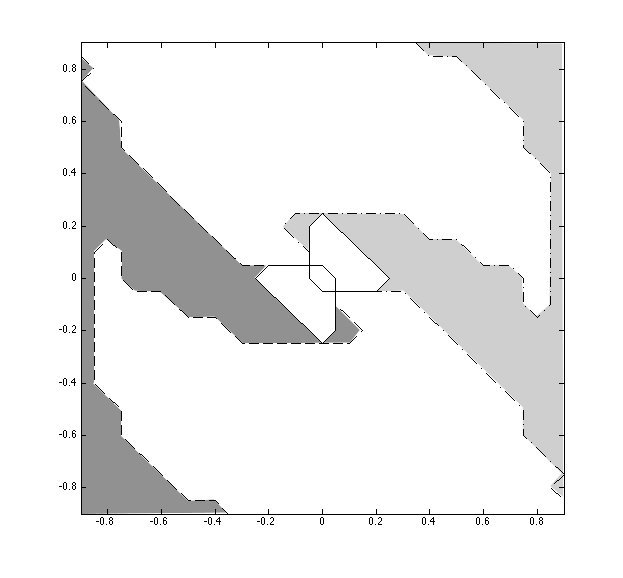}
\vspace{-0.5cm}
\caption{Van Der Pol: exact decomposition and two (of the four) approximated independent subsets found with a coarse grind of $15^2$ points.} \label{fig2}
\end{center}
\end{figure}
\end{example}

\begin{table}
\begin{center}
\caption{Van Der Pol: comparison of the accuracy of the decomposition with various discretization steps}\label{tt:4}
\begin{tabular}{{c}|*{2}{c}|{c}|{c}}
         N. of variables  & $\Delta x_K$  &  Time elapsed   &   $\max_i |\overline{\Sigma}_i|/|\Omega|$ & $\max_i |\Sigma_i|/|\Omega|$\\
\hline
\bfseries $5^2$ & 0.4 &  $1.4\cdot 10^{-3}$ & $62\%$   &   \\
\bfseries $10^2$ & 0.2 &  $0.011$ &  $55\%$ &\\
\bfseries $20^2$ & 0.1 & $0.103$ & $47\%$ &  42.2\%\\
\bfseries $30^2$  & 0.06 &  $1.47$ &  $45\%$  &\\
\bfseries $40^2$  & 0.05 &  $5.6 $  &  $44.6\%$  &   \\
\bfseries $50^2$ & 0.04 & $16.3$ & $44.1\%$ &  \\
\end{tabular}
\end{center}
\end{table}

\begin{table}
\begin{center}
\caption{Van Der Pol: approximation error Error in norm $\Delta\infty$ (and $\Delta_1$) }\label{tt:5}
\begin{tabular}{{c}|*{3}{c}}
             &  $50^2 $ &  $100^2$  &    $200^2$\\
\hline
\bfseries original  &  $0.09 (0.07)$ &$0.03 (0.01)$  &  $0.01 (6\cdot 10^{-3})$  \\
\bfseries 2-subsets  &  $0.09 (0.07)$ &$0.03 (0.01)$  &  $0.01 (6\cdot 10^{-3})$    \\
\bfseries  4-subsets &  $0.09 (0.07)$ &$0.03 (0.01)$  &  $0.01 (6\cdot 10^{-3})$   \\
\bfseries 8-subsets  &  $0.09 (0.07)$ &$0.03 (0.01)$  &  $0.01 (6\cdot 10^{-3})$  \\
\end{tabular}
\end{center}
\end{table}

We consider a division of the target in ``slices of a cake'', meaning that a 2-parts division will be $\Gamma_1:=\{x\in B(0,0.2)| x_2\geq 0\}$, $\Gamma_2:=\{x\in B(0,0.2)| x_2\leq 0\}$ and a 4-parts, $\Gamma_1:=\{x\in B(0,0.2)| x_1\geq 0, x_2\geq 0\}$, $\Gamma_2:=\{x\in B(0,0.2)| x_1\geq 0,  x_2\leq 0\}$, $\Gamma_3:=\{x\in B(0,0.2)| x_1\leq 0, x_2\leq 0\}$, $\Gamma_4:=\{x\in B(0,0.2)| x_1\leq 0,  x_2\geq 0\}$, etc.\par
 The estimation of the constants will be in this case less elementary: we will choose in this example $C=1$, $M=1$, $\epsilon=10^{-3}$, $q=\frac{3}{4}$.\par
 In Figure \ref{fig2} is shown a comparison between the exact division in sub-domains and two approximated sets ($\overline{\Sigma}_1$, $\overline{\Sigma}_3$). An additional difficulty in this case consists in the fact that there are some points of the domain not reachable by the trajectories without exiting from the computational domain (in white in the l.h.s. figure). These points are automatically included in each approximation  but, although they contribute to enlarge every subset, they do not create serious problems to the procedure.\par
In Table \ref{tt:4} is shown the accuracy of the 4-independent subset reconstruction with various discretization steps. In this case it is possible to see an innate limitation of the efficacy of such decomposition for  parallel computing: the exact division in independent subset is not balanced, then the reduction of dimension in the greater subset will be less considerable. In some  cases this problem could even nullify the efficacy of the method (we can obtain a decomposition in some empty sets and the whole $\Omega$), we will discuss this point in the conclusions section.\par
As already remarked this decomposition guarantees a rate of convergence to the solution equal to the resolution on the whole domain. The experimental evidence in Table \ref{tt:5} shows something more, with an error (both in $\Delta_\infty$ than in $\Delta_2$ norm) constant, in its significant figures, in the various cases of decomposition. This is the typical situation: simply the decomposition does not affect the convergence of the numerical method.

\begin{example}[A pursuit-evasion game]\label{ex4}
Let us now to pass to a decomposable differential game. We will consider a  pursuit evasion game, where two agents have the opposite goal to reduce/postpone the time of capture. The dynamics considered are the following:

\begin{figure}[t]
\begin{center}
\includegraphics[height=5.2cm]{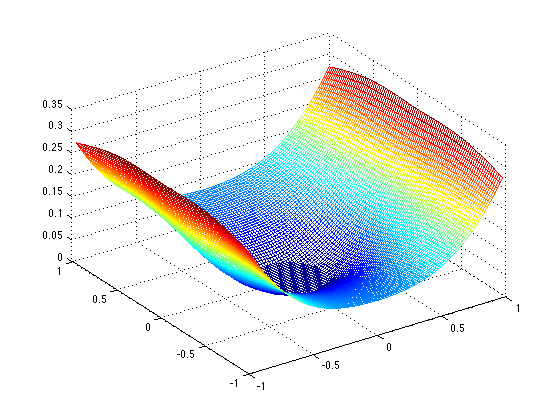}
\includegraphics[height=5.2cm]{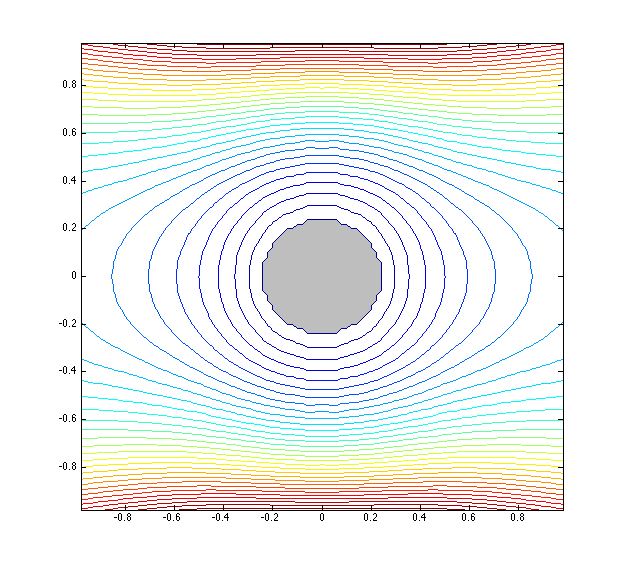}
\vspace{-0.5cm}
\caption{A pursuit evasion game: approximated value function of the differential game presented in Example \ref{ex4}} \label{fig4}
\includegraphics[height=5.6cm]{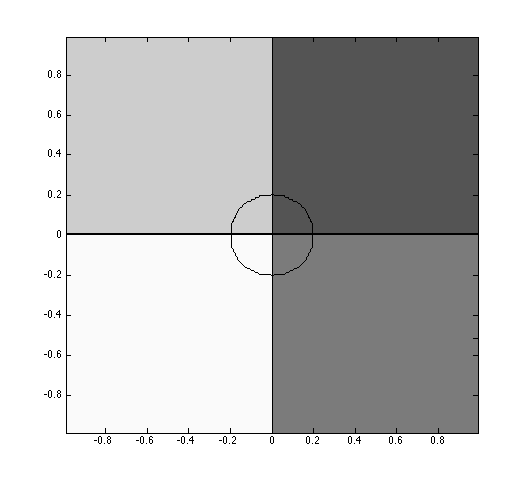}
\includegraphics[height=5.6cm]{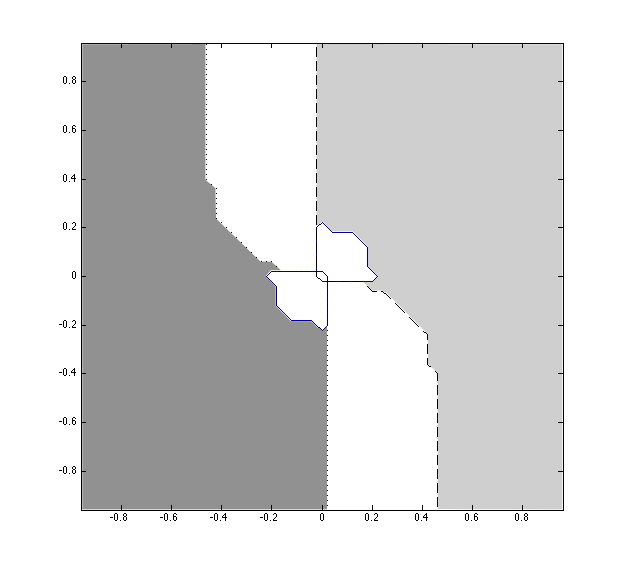}
\vspace{-0.5cm}
\caption{A pursuit evasion game: exact decomposition and two (of the four) approximated independent subsets found with a coarse grind of $40^2$ points.} \label{f:pe}
\end{center}
\end{figure}

\begin{equation*}
  f(x,a,b) :=
\left(\begin{array}{c}
f_1(x) (a_1-b_1)\\
f_2(x) (a_2-b_2) \end{array} \right)
\end{equation*}
where the functions $f_1$, $f_2$ are $f_1(x):=x_2+1$ and $f_2(x):=1$. The running cost $l(x,a,b):=x_1+0.1$. This is a modification of the classical pursuit evasion game on a plane to put in evidence another aspect of the technique, as will be clear in the following. The controls are taken in the unit ball for the pursuer $A=B(0,1)$ and $B=B(0,1/2)$ for the evader. The capture happens when the trajectory is driven to touch the small ball $B(0,\rho)$, ($\rho = 0.2$, in this case), then the set $\Gamma$ will be, as in the previous example $\Gamma:=\partial B(0,0.2)$.\par

\begin{table}
\begin{center}
\caption{A pursuit evasion game: comparison of the accuracy of the decomposition with various discretization steps}\label{tt:5}
\begin{tabular}{{c}|*{2}{c}|{c}|{c}}
         N. of variables  & $\Delta x_K$  &  Time elapsed   &   $\max_i |\overline{\Sigma}_i|/|\Omega|$ & $\max_i |\Sigma_i|/|\Omega|$\\
\hline
\bfseries $5^2$ & 0.4 &  $ 10^{-3}$ & $60\%$   &   \\
\bfseries $10^2$ & 0.2 &  $0.008$ &  $46\%$ &\\
\bfseries $30^2$  & 0.06 &  $1.38$ &  $38\%$  &25\%\\
\bfseries $50^2$ & 0.04 & $15.9$ & $36.1\%$ &  \\
\end{tabular}
\end{center}
\end{table}

It is possible to show that the Hamilton-Jacobi equation associated to this problem verifies the decomposability condition \eqref{C)}; the easiest way to do it is to consider the norm (it is possible because $|f_i(x)|>0$ for $i=1,2$)
$$ \|p\|_*:=\max_{a\in B(0,1)} \left(\begin{array}{c} f_1(x)\\ f_2(x) \end{array} \right) a^T\cdot p, $$
the Hamiltonian associated is equivalent to 
$$H(x,p):=\|p\|_*-\frac{\|p\|_*}{2}-(x_1^2+0.1)$$
evidently convex everywhere with respect to $p$; then \eqref{C)} is automatically verified.\par
The value function of the game is shown in Figure \ref{fig4}. The main point is that the function is very flat along the axis $x_2=0$, this will produce a critical effect in the sets approximation, shown in Figure \ref{f:pe} and in Table \ref{tt:5}, in this test, the parameter are fixed as $C=1$, $M=3$, $q=1/2$. The convergence to the exact division in independent subsets will be very slow.\par
%It should be natural to choose an arbitrary split in the regions of overlapping of the sets $\overline{\Sigma}_i$. In Figure \ref{} on the right there is a possible proposal (we avoid the lucky choice of the correct decomposition). In Table is shown the error using the two possible decompositions. Accordingly with the theory, the partition produce an error of order $O(\Delta x_K)$, instead the (ISA), as in the previous tests, preserve an error of order $O(\Delta x_G)$.
\end{example}

\section{Conclusions}
We have shown a constructive manner to obtain a decomposition of the domain of a Hamilton Jacobi equation verifying condition \eqref{C)} in independent subsets  which have the property of being computed independently from each others. The procedure resumes some general ideas already presented in \cite{cacace2012patchy}, clarifying the theoretical background, enlarging the class of the Hamilton-Jacobi equations where the technique is relevant, proving the convergence of the parallel algorithm derived ISA and producing some estimates for the error. \par
A detailed evaluation of the performances of ISA are still an open question postponed to a forthcoming work. Despite that, we can expect results similar to \cite{cacace2012patchy}, since as shown in the Section \ref{test}, our pre-computing step gives a division in sub-domains sufficiently close to a partition. Anyway, the main advantage of our proposal is in the guarantee to converge to the correct solution with an error of the order $O(\Delta x_G)$.
\par

Some further improvements can be adapted to the technique. The critical occurrence shown in Example \ref{ex3}, about the balance of the dimension of the subsets can be solved with a recursive refinement of the division of $\Gamma$, producing in some few steps a more equilibrate division. More unavoidable the case presented in example \ref{ex4}. In that case it is, for the moment, impossible to obtain a satisfactory reduction of the dimension of the decomposed domains without solving the problem on a sufficiently fine grid. 

%An immediate possible solution should be consider a collection of smaller sets, or even a partition of the domain. In such case, however, we should accept, as remarked in Corollary \ref{Patchy}, the possibility of an error of order $O(\sqrt{\Delta x_K})$, which does not legitimize the resolution on the fine grid $G$.

\section*{Acknowledgements}
This work was partially supported by the European Union under the 7th Framework Programme FP7-PEOPLE-2010-ITN SADCO, Sensitivity Analysis for Deterministic Controller Design.

\def\cprime{$'$} \def\cprime{$'$} \def\cprime{$'$} \def\cprime{$'$}
  \def\cprime{$'$}

\end{document}